%% file: Lagrangian_capacity.tex
\documentclass[a4paper,12pt]{article}

\usepackage[utf8]{inputenc}                  %
\usepackage[T1]{fontenc}                     %
\usepackage{lmodern}                         %
\usepackage{geometry}                        
\usepackage[square,numbers]{natbib}          
\usepackage[nottoc,notlot,notlof]{tocbibind} 
\usepackage{enumitem}                        
\usepackage{xparse}                          
\usepackage{xstring}                         
\usepackage{etoolbox}                        
\usepackage{parskip}                         
\usepackage{titling}

\usepackage{mathtools}              
\usepackage{amssymb}                
\usepackage{amsthm}                 
\usepackage{IEEEtrantools}          
\usepackage{tikz}                   

\usepackage{hyperref}               
\usepackage{bookmark}               
\usepackage[capitalise]{cleveref}   
\usepackage[all]{hypcap}            

\apptocmd{\sloppy}{\hbadness 10000\relax}{}{}                       
\renewcommand\theequationdis{\normalfont\normalcolor(\theequation)} 
\allowdisplaybreaks                                                 
\graphicspath{{./figures/}}                                         

\newlength{\alphabet}
\setlist[description]{font=\normalfont}
\setlist[enumerate]{font=\normalfont}
\setlist[enumerate,1]{label = {(\arabic*)}}
\setlist[enumerate,2]{label = {(\arabic{enumi}.\arabic*)}}

\newcounter{dummy}
\makeatletter
\newcommand\myitem[1][]{\item[#1]\refstepcounter{dummy}\def\@currentlabel{#1}}
\makeatother

\usetikzlibrary{decorations.pathreplacing} 
\usetikzlibrary{math}                      
\usetikzlibrary{calc}                      
\usetikzlibrary{cd}                        
\tikzset{
    symbol/.style={
        draw=none,
        every to/.append style={
            edge node={node [sloped, allow upside down, auto=false]{$#1$}}
        },
    },
}

\hypersetup{
    bookmarksnumbered=true,
    colorlinks=true,
    linkcolor=blue,
    citecolor=blue,
    urlcolor=blue
}

\theoremstyle{plain}      
\newtheorem{theorem}     {Theorem} 
\newtheorem{proposition} [theorem] {Proposition}
\newtheorem{lemma}       [theorem] {Lemma}

\newtheorem{conjecture}  [theorem] {Conjecture}

\theoremstyle{definition} 

\newtheorem{definition}  [theorem] {Definition}
\newtheorem{example}     [theorem] {Example}
\newtheorem{remark}      [theorem] {Remark}

\NewDocumentCommand{\N}{}{\mathbb{N}} 
\NewDocumentCommand{\Z}{}{\mathbb{Z}} 
\NewDocumentCommand{\Q}{}{\mathbb{Q}} 
\NewDocumentCommand{\R}{}{\mathbb{R}} 
\NewDocumentCommand{\C}{}{\mathbb{C}} 

\title{Cube normalized symplectic capacities}
\author{Jean Gutt \and Miguel Pereira \and Vinicius G. B. Ramos}
\date{}

\begin{document}

\maketitle

\begin{abstract}
	We introduce a new normalization condition for symplectic capacities, which we call cube normalization. This condition is satisfied by the Lagrangian capacity and the cube capacity. Our main result is an analogue of the strong Viterbo conjecture for monotone toric domains in all dimensions. Moreover, we give a family of examples where standard normalized capacities coincide but not cube normalized ones. Along the way, we give an explicit formula for the Lagrangian capacity on a large class of toric domains.
\end{abstract}

\section{Introduction}
The study of symplectic embeddings is at the core of symplectic geometry. One of the most important tools in this study are symplectic capacities. A {\bf symplectic capacity} is a function which assigns to each symplectic manifold $(X,\omega)$ of a fixed dimension $2n$, possibly in some restricted class, a number $c(X,\omega)$ satisfying the following conditions:
\begin{enumerate}
	\item If there exists an embedding $\varphi:X_1\hookrightarrow X_2$ such that $\varphi^*\omega_2=\omega_1$, then
	\[c(X_1,\omega_1)\le c(X_2,\omega_2).\]
	\item If $r>0$, then
	\[c(X,r\cdot \omega)=r \cdot c(X,\omega).\]
\end{enumerate}
After Gromov's seminal work on symplectic embeddings \cite{Gromov}, many capacities were defined. The majority of these satisfy a normalization condition based on Gromov's non-squeezing. More precisely, let $B^{2n}(r)\subset\C^n$ denote the ball of radius $r$ and let $Z^{2n}(r)=B^{2}(r)\times\C^{n-1}$. As usual, the standard symplectic form on $\C^n (=\R^{2n})$ is defined by
\[\omega_0=\sum_{i=1}^n dx_i\wedge dy_i.\]
We say that a symplectic capacity is {\bf ball normalized}\footnote{A capacity satisfying condition 3) is usually called {\bf normalized} in the literature. We add the word ``ball'' in this paper because we will define another normalization condition below.} if
\begin{itemize}
\item[(3)] $c(B^{2n}(r),\omega_0)=c(Z^{2n}(r),\omega_0)=\pi r^2$.
\end{itemize}
The central question about ball normalized capacities is the following conjecture, which apparently has been folkore since the 1990s.
\begin{conjecture}[strong Viterbo conjecture]
\label{conj:V}
If $X$ is a convex domain in $\R^{2n}$, then all normalized symplectic capacities of $X$ are equal.
\end{conjecture}
We refer to \cite{GHR} for a presentation of known results around Conjecture \ref{conj:V}. The strong Viterbo conjecture is, in particular, proven for all monotone toric domains in dimension 4.

Some of the main examples of symplectic capacities that do not satisfy this ball normalization 3) come in sequences, see \cite{EH2,qech}. For all of these sequences, the first capacity is still ball normalized. Two other capacities stand alone not satisfying 3), namely the {\bf Lagrangian capacity} and the {\bf cube capacity}, defined by Cieliebak--Mohnke \cite{CM} and Gutt--Hutchings \cite{GH}, respectively. In this paper we will introduce a new normalization condition (which we call cube normalization) which is satisfied by these latter capacities. Our main result is an equivalent of the strong Viterbo conjecture for cube normalized capacities.
\begin{theorem}\label{thm:normalized}
	All cube normalized symplectic capacities coincide on all monotone toric domain in any dimension.
\end{theorem}

This paper is organized as follows. In Section 2, we define the cube normalization and prove Theorem \ref{thm:normalized}. In Section 3, we provide an explicit formula for the Lagrangian capacity on a large class of toric domains encompassing monotone toric domains. In Section 4, we study cube normalized capacities of an interesting class of examples of non-monotone toric domains and we show that for some parameters, ball normalized capacities coincide while cube normalized do not. Finally, in Section 5, we find an upper bound for the cube capacity of a large class of weakly convex toric domains, which is used in Section 4.

\section{A new normalization condition}
Given a domain\footnote{In this article, a domain is the closure of a non-empty open set.} $\Omega\subset\R^n_{\ge 0}$, define the toric domain
\[X_\Omega=\mu^{-1}(\Omega)=\left\{(z_1,\dots,z_n)\in\C^n\mid (\pi|z_1|^2,\dots,\pi|z_n|^2)\in\Omega\right\}\] 
where the map $\mu:\C^n\to[0,+\infty)^n\,:\,(z_1,\ldots, z_n)\mapsto(\pi|z_1|^2,\dots,\pi|z_n|^2)$ is the periodic moment map.
We let \[\partial_+\Omega=\left\{p=(p_1,\dots,p_n)\in\partial \Omega\mid p_i>0 \text{ for }i=1,\dots,n.\right\}.\] Recall from \cite{GHR} that a {\bf monotone toric domain} is a compact toric domain with smooth boundary such that for every $p\in\partial_+\Omega$, the outward pointing normal vector at $p$, $\nu=(\nu_1,\dots,\nu_n)$ satifies $\nu_i\ge 0$ for $i=1,\dots,n$. Note that a monotone toric domain is the limit of toric domains $X_{\Omega'}$ where $\Omega'$ is bounded by the coordinate hyperplanes and the graph of a function whose partial derivatives are all negative, see the proof of \cite[Lemma 3.2]{GHR}.

Consider the following examples of toric domains:
\begin{IEEEeqnarray*}{lrClCrCl}
    \textrm{The Ball}     \quad & B_n(a) & \coloneqq & \mu^{-1}(\Omega_{B_n(a)}), & \quad & \Omega_{B_n(a)} & \coloneqq & \{ x \in \R^n_{\geq 0} \mid x_1 + \cdots + x_n \leq a \}, \\
    \textrm{The Cylinder} \quad & Z_n(a) & \coloneqq & \mu^{-1}(\Omega_{Z_n(a)}), & \quad & \Omega_{Z_n(a)} & \coloneqq & \{ x \in \R^n_{\geq 0} \mid x_1 \leq a \}, \\
    \textrm{The Cube}     \quad & P_n(a) & \coloneqq & \mu^{-1}(\Omega_{P_n(a)}), & \quad & \Omega_{P_n(a)} & \coloneqq & \{ x \in \R^n_{\geq 0} \mid \forall i = 1, \ldots, n \colon x_i \leq a \}, \\
    \textrm{The NDUC}     \quad & N_n(a) & \coloneqq & \mu^{-1}(\Omega_{N_n(a)}), & \quad & \Omega_{N_n(a)} & \coloneqq & \{ x \in \R^n_{\geq 0} \mid \exists i = 1, \ldots, n \colon x_i \leq a \}.
\end{IEEEeqnarray*}
Here, NDUC stands for non-disjoint union of cylinders.
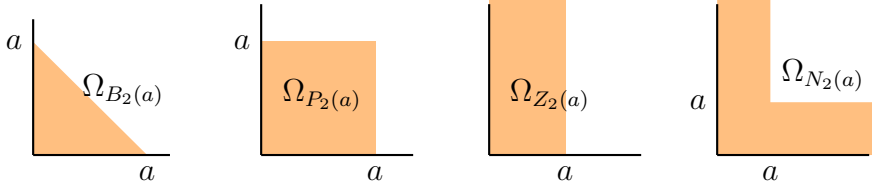
\begin{figure}[ht]
    \centering
\begin{tikzpicture}
	\fill[orange!50](0,0)--(0,1.5)--(1.5,0)--(0,0);
	\draw [thick](0,0)--(0,1.8) ;
	\draw [thick] (0,0)--(1.8,0) ;
	\draw (0,1.5) node[left]{$a$};
	\draw (1.5,0) node[below]{$a$};
	\draw (0.5,0.5) node[above right]{$\Omega_{B_2(a)}$};
	\fill[orange!50](3,0)--(3,1.5)--(4.5,1.5)--(4.5,0)--(3,0);
	\draw [thick](3,0)--(3,2) ;
	\draw [thick] (3,0)--(5,0) ;
	\draw (3,1.5) node[left]{$a$};
	\draw (4.5,0) node[below]{$a$};
	\draw (3.8,0.8) node{$\Omega_{P_2(a)}$};
	\fill[orange!50](6,0)--(6,2.1)--(7,2.1)--(7,0)--(6,0);
	\draw [thick](6,0)--(6,2) ;
	\draw [thick] (6,0)--(8,0) ;
	\draw (7,0) node[below]{$a$};
	\draw (6.8,0.8) node{$\Omega_{Z_2(a)}$};
	\fill[orange!50](9,0)--(9,2.1)--(9.7,2.1)--(9.7,0.7)--(11.1,0.7)--(11.1,0)--(9,0);
	\draw [thick](9,0)--(9,2) ;
	\draw [thick] (9,0)--(11,0) ;
	\draw (9.7,0) node[below]{$a$};
	\draw (9,0.7) node[left]{$a$};
	\draw (9.7,0.7) node[above right]{$\Omega_{N_2(a)}$};
\end{tikzpicture}
\caption{The domains $\Omega$ for the aforementioned domains for $n=2$}
\end{figure}
Within those toric domains, the ball normalization condition reformulates as
\[
	c\big(B_n(1)\big)=c\big(Z_n(1)\big)=1.
\]
This normalization stemmed out of Gromov's non-squeezing theorem \cite{Gromov} asserting that there exists a symplectic embedding $B_n(a)\hookrightarrow Z_n(b)$ if and only if $a\leq b$. The first examples of normalized capacities are the
\textbf{Gromov width} $c_B$ and the {\bf cylindrical capacity} $c^Z$ defined for any symplectic manifold $(X, \omega)$.
\begin{IEEEeqnarray*}{rCll}
    c_B(X,\omega) & \coloneqq & \sup & \{ a \mid \text{ there exists a symplectic embedding } B_n(a) \longrightarrow X \}, \\
    c^Z(X,\omega) & \coloneqq & \inf & \{ a \mid \text{ there exists a symplectic embedding } X \longrightarrow Z_n(a) \},
\end{IEEEeqnarray*}
Additional examples of normalized symplectic capacities are the Hofer-Zehnder capacity $c_{\textrm{HZ}}$ defined in \cite{HZ} and the Viterbo capacity $c_{\textrm{SH}}$ defined in \cite{V}. There are also useful families of symplectic capacities parametrized by a positive integer $k$ including the Ekeland-Hofer capacities $c_k^{\textrm{EH}}$ defined in \cite{EH,EH2} using calculus of variations; the ``equivariant capacities'' $c_k^{\textrm{GH}}$ defined in \cite{GH} using positive equivariant symplectic homology; and in the four-dimensional case, the ECH capacities $c_k^{\textrm{ECH}}$ defined in \cite{qech} using embedded contact homology. For each of these families, the $k=1$ capacities $c_1^{\textrm{EH}}$, $c_1^{\textrm{CH}}$, and $c_1^{\textrm{ECH}}$ are normalized. For more about symplectic capacities in general we refer to \cite{chls, schlenk} and the references therein.

We now introduce a new normalization based on a ``non-squeezing theorem'' for the cube.
\begin{theorem}[{\cite[Proposition 1.20]{GH}}]
    \label{thm:gh cube capacity}
    There exists a symplectic embedding $P_n(a)\hookrightarrow N_n(b)$ if and only if $a\leq b$.
\end{theorem}
This theorem, together with the previous discussion, motivates the following definition.
\begin{definition}
	We say that a symplectic capacity $c$ is \textbf{cube normalized} if
	\begin{IEEEeqnarray*}{c+x*}
        c(P_n(1)) = c(N_n(1)) = 1.
    \end{IEEEeqnarray*}
\end{definition}
We now wish to present examples of cube normalized symplectic capacities.

The first example is the {\bf cube capacity} $c_P$ \cite{GH}
\[
	c_P(X,\omega) := \sup\{ a \mid \text{ there exists a symplectic embedding } P_n(a) \longrightarrow X \},
\]
A second example is the {\bf NDUC capacity} $c^N$
\[
	c^N(X,\omega) := \inf\{ a \mid \text{ there exists a symplectic embedding } X \longrightarrow N_n(a) \}
\]

The first non immediate example of a cube normalized symplectic capacity was introduced by Cieliebak and Mohnke \cite{CM} and proved to be cube normalized by the second author in his PhD \cite{Per}.
Let $(X, \omega)$ be a symplectic manifold and let $L \subset X$ be a Lagrangian submanifold. The \textbf{minimal area} of $L$ is given by
\begin{IEEEeqnarray*}{c+x*}
    A_{\mathrm{min}}(L) \coloneqq \inf 
    \Big\{ \int_\sigma \omega \ \Big|\ \sigma \in \pi_2(X, L),\, \int_\sigma\omega > 0 \Big\}.
\end{IEEEeqnarray*}
The \textbf{Lagrangian capacity} of $(X,\omega)$ is defined as
\begin{IEEEeqnarray*}{c+x*}
    c_L(X,\omega) \coloneqq \sup \{ A_{\mathrm{min}}(L) \mid L \text{ is an embedded Lagrangian torus} \}. 
\end{IEEEeqnarray*}
\begin{theorem}[\cite{Per}]
	\[
		c_L(P_n(1)) = c_L(N_n(1)) = 1.
	\]
\end{theorem}
The second author actually proved a stronger result. For any toric domain $X_{\Omega} \subset \C^n$, define its \textbf{diagonal} to be
\begin{IEEEeqnarray*}{c+x*}
    \delta_{\Omega} \coloneqq \sup \{ a \mid (a, \ldots, a) \in \Omega \}.
\end{IEEEeqnarray*}
\begin{theorem}[{\cite[Theorem 7.65]{Per}}]
    \phantomsection\label{thm:lag cap convex concave}
    If $X_\Omega$ is a convex or concave toric domain then%
    \begin{IEEEeqnarray*}{c+x*}
        c_L(X_{\Omega}) = \delta_\Omega.
    \end{IEEEeqnarray*}
\end{theorem}
%

\begin{remark}
    The proof of \cref{thm:lag cap convex concave} uses linearized contact homology, and this result is stated under some assumptions about this theory. For a more detailed discussion on these assumptions see \cite[Disclaimer 1.11]{Sie} and \cite[Section 7.1]{Per}.
\end{remark}

\begin{remark}
    \label{exa:other are cube normalized}
    The proof of \cref{thm:lag cap convex concave} uses other symplectic capacities, namely
    \begin{enumerate}
        \item the \textbf{Gutt--Hutchings capacities} from \cite{GH}, denoted by $c^{\mathrm{GH}}_k$;
        \item the \textbf{higher symplectic capacities} from \cite{Sie}, denoted by $\mathfrak{g}_k^{\leq 1}$;
        \item the \textbf{McDuff--Siegel capacities} from \cite{MS}, denoted by $\tilde{\mathfrak{g}}_k^{\leq 1}$.
    \end{enumerate}
    Inspecting the proof of this theorem, one sees that the proof extends word for word for any monotone toric domain, and that moreover
    \begin{IEEEeqnarray*}{c+x*}
        c_L(X_{\Omega}) 
        = \lim_{k \to +\infty} \frac{\tilde{\mathfrak{g}}_k^{\leq 1}(X_{\Omega})}{k} 
        = \lim_{k \to +\infty} \frac{{\mathfrak{g}}_k^{\leq 1}(X_{\Omega})}{k} 
        = \lim_{k \to +\infty} \frac{c^{\mathrm{GH}}_k(X_{\Omega})}{k}
        = \delta_{\Omega}
    \end{IEEEeqnarray*}
    for any monotone toric domain $X_{\Omega}$.
    \end{remark}
    One can therefore define cube normalized symplectic capacities as follows.
    \begin{definition}
    For a nondegenerate Liouville domain $(X,\lambda)$, let
    \begin{IEEEeqnarray*}{rCls+x*}
        c^{\mathrm{GH}}_{\inf}(X)               & \coloneqq & \liminf_{k} \frac{c_k^{\mathrm{GH}}(X)}{k}, \\
        \mathfrak{g}_{\inf}^{\leq 1}(X)         & \coloneqq & \liminf_{k} \frac{\mathfrak{g}_{k}^{\leq 1}(X)}{k}, \\
        \tilde{\mathfrak{g}}_{\inf}^{\leq 1}(X) & \coloneqq & \liminf_{k} \frac{\tilde{\mathfrak{g}}_{k}^{\leq 1}(X)}{k}.
    \end{IEEEeqnarray*}
    By Remark \ref{exa:other are cube normalized} the symplectic capacities $c^{\mathrm{GH}}_{\inf}$, $\mathfrak{g}_{\inf}^{\leq 1}$ and $\tilde{\mathfrak{g}}_{\inf}^{\leq 1}$ are cube normalized.
    \end{definition}
    Using the main result of \cite{GR} asserting that for all $k\geq1$ $c_k^{\mathrm{GH}}=c_k^{\mathrm{EH}}$, we have another cube normalized symplectic capacity
    \[
    	c^{\mathrm{EH}}_{\inf}(X) := \liminf_{k} \frac{c_k^{\mathrm{EH}}(X)}{k}.
    \]
	Note that the main result of \cite{GR} together with Remark \ref{exa:other are cube normalized} shows that for any monotone toric domain $X_\Omega$
	\[
		c_L(X_\Omega)=\lim_{k \to +\infty} \frac{c^{\mathrm{EH}}_k(X_{\Omega})}{k}.
	\]
	This answers (for the monotone toric case) a Question by Cieliebak-Mohnke \cite{CM} who asks whether this equality holds for all convex domains in $\R^{2n}$.
	
The following theorem, which is an analogue of Viterbo's strong conjecture is our main result:

\begin{theorem}
    \label{thm:cube normalized}
    All cube normalized capacities coincide on monotone toric domains in $\R^{2n}$.
\end{theorem}
\begin{proof}
    Let $c$ be a cube normalized symplectic capacity and let $X_{\Omega}$ be a monotone toric domain in $\R^{2n}$. We are going to show that then the value of $c(X_\Omega)$ is determined.
    The monotonicity of $X_{\Omega}$ ensures that
    \[P_n(\delta_{\Omega}) \subset X_{\Omega} \subset N_n(\delta_{\Omega}).\] Then,
    \begin{IEEEeqnarray*}{rCls+x*}
        \delta_{\Omega}
        & =    & c(P_n(\delta_{\Omega})) & \quad [\text{since $c$ is cube normalized}] \\
        & \leq & c(X_{\Omega})           & \quad [\text{by monotonicity}] \\
        & \leq & c(N_n(\delta_{\Omega})) & \quad [\text{by monotonicity}] \\
        & =    & \delta_{\Omega}         & \quad [\text{since $c$ is cube normalized}].  & \qedhere
    \end{IEEEeqnarray*}
\end{proof}
As a corollary of \cref{thm:cube normalized}, we have the following formula for the value of cube normalized symplectic capacities on monotone toric domains.
\begin{theorem}
	Let $c$ be a cube normalized symplectic capacity and let $X_{\Omega}$ be a monotone toric domain in $\R^{2n}$. Then
	\[
		c(X_{\Omega}) = \delta_{\Omega}.
	\]
\end{theorem}
In view of \cref{thm:cube normalized}, it is reasonable to conjecture the following:
\begin{conjecture}
	All cube normalized capacities coincide on convex domains in $\C^n$.
\end{conjecture}

We wish now to make a few comments on what precedes:
\begin{remark}
	The link between monotone toric and convex is studied intensively and is, at the moment, unclear. All monotone toric domains are dynamically convex\footnote{Convexity is not a symplectically invariant property. This was already pointed out a long time ago but only a few symplectic substitutions have been suggested. The most prominent one is \textbf{dynamical convexity}, introduced in \cite{HWZ2}, where they show that strict convexity guarantees dynamical convexity.} toric domains; however the converse is only true in $\R^4$. Examples of monotone toric domains not symplectomorphic to a convex domain where produced recently \cite{DGZ, CE}.
\end{remark}
\begin{remark}
    If $c$ is a cube normalized symplectic capacity, then $c$ is not normalized in the usual sense. Indeed, by \cref{thm:cube normalized}, if $c$ is cube normalized then $c(B_n(1)) = 1/n$ and $c(Z_n(1)) = 1$.
    We have the following inequalities (for any 2n-dimensional symplectic manifold $(X,\omega)$):
    \[
		c_P(X,\omega)\leq c_B(X,\omega)\leq nc_P(X,\omega).
    \]
    Those inequalities come from the optimal embeddings
    \[
    	B_n(a)\subset P_n(a)\subset B_n(na)
    \]
    We also have
    \[
    		c^N(X,\omega)\leq c^Z(X,\omega)
    \]
    coming from the inclusion $Z_n(a)\subset N_n(a)$.
    \begin{conjecture}
    		\[c^Z(X,\omega)\leq nc^N(X,\omega).\]
    \end{conjecture}
    The conjecture is true for $n=2$. This is the main technical point of \cite{GHR}. This amounts to prove that there exists a symplectic embedding
    \[
    		N_n(a)\hookrightarrow Z_n(na).
    \]
\end{remark}
\begin{remark}
	The minimal area of a Lagrangian torus, $A_{\mathrm{min}}(L)$, is not continuous in $L$. Indeed on a toric domain $X_\Omega$, $\mu^{-1}(x)$ is a Lagrangian torus for $x=(x_1,\ldots,x_n)\in(\mathrm{int}\Omega\cup\partial_+\Omega)$.
	By \cref{lem:a min with exact symplectic manifold},
	\begin{equation}\label{eq:Amin}
		A_{\mathrm{min}}\big(\mu^{-1}(x)\big)=\inf\{k_1x_1+\cdots+k_nx_n\,|\,k_1,\ldots k_n\in\Z\}.
	\end{equation}
\end{remark}

\section{Computing of the Lagrangian capacity for a more general family of toric domains}

In this section, we will see how one can use \cref{thm:lag cap convex concave} to compute the Lagrangian capacity for a larger class of toric domains which are not necessarily monotone (see \cref{thm:lag cap any toric} below). For a toric domain $X_{\Omega}$, define
\begin{IEEEeqnarray*}{c+x*}
    \eta_{\Omega} \coloneqq \inf \{ a \mid X_{\Omega} \subset N_n(a) \}.
\end{IEEEeqnarray*}
Notice that if $X_{\Omega}$ is convex or concave, then $\delta_{\Omega} = \eta_{\Omega}$. To prove \cref{thm:lag cap any toric}, we will make use of the following lemma:

\begin{lemma}[{\cite[Lemma 6.16]{Per}}]
    \label{lem:a min with exact symplectic manifold}
    Let $(X,\lambda)$ be an exact symplectic manifold and $L \subset X$ be a Lagrangian submanifold. If $\pi_1(X) = 0$, then
    \begin{IEEEeqnarray*}{c+x*}
        A _{\mathrm{min}}(L) = \inf \left\{ \lambda(\rho) \ | \ \rho \in \pi_1(L), \ \lambda(\rho) > 0 \right\}.
    \end{IEEEeqnarray*}
\end{lemma}
\begin{proof}
    The diagram
    \begin{IEEEeqnarray*}{c+x*}
        \begin{tikzcd}
            \pi_2(X,L) \ar[dr, swap, "\omega"] \ar[r, two heads,"\partial"] & \pi_1(L) \ar[d, "\lambda"] \ar[r, "0"] & \pi_1(X) \\
            & \R
        \end{tikzcd}
    \end{IEEEeqnarray*}
    commutes, where $\partial([\sigma]) = [\sigma|_{S^1}]$, and the top row is exact.
\end{proof}

\begin{theorem}
    \label{thm:lag cap any toric}
    Let $X_{\Omega}$ be a toric domain. If $(\eta_{\Omega},\ldots,\eta_{\Omega}) \in \partial \Omega$ then%
    \begin{IEEEeqnarray*}{c+x*}
        c_L(X_{\Omega}) = \eta_{\Omega}.
    \end{IEEEeqnarray*}
\end{theorem}
\begin{proof}
    By definition of $\eta_{\Omega}$, we have $X_{\Omega} \subset N_n(\eta_{\Omega})$. Define $T \coloneqq \mu^{-1}(\eta_{\Omega},\ldots,\eta_{\Omega})$. Then $T$ is an embedded Lagrangian torus in $X_{\Omega}$ (see \cref{fig:main} for an illustration of $\eta_{\Omega}$, $T$, $\Omega$ and $\Omega_{N_n(\eta_{\Omega})}$). Therefore,
    \begin{IEEEeqnarray*}{rCls+x*}
        \eta_{\Omega}
        & =    & A_{\mathrm{min}}(T)     & \quad [\text{by \cref{lem:a min with exact symplectic manifold}}] \\
        & \leq & c_L(X_{\Omega})         & \quad [\text{by definition of $c_L$}] \\
        & \leq & c_L(N_n(\eta_{\Omega})) & \quad [\text{by monotonicity}] \\
        & \leq & \eta_{\Omega}           & \quad [\text{by \cref{thm:lag cap convex concave}}].                & \qedhere
    \end{IEEEeqnarray*}
\end{proof}
Note that \cref{thm:lag cap any toric} extends mutatis mutandis, using \cref{eq:Amin},  to the following
\begin{theorem}\label{thm:generalformulacL}
	Let $X_\Omega\subset N_n(\eta_{\Omega})$ be a toric domain in $\R^{2n}$ such that there exist a point $x\in\overline{\partial_+\Omega}\cap \partial_+N_n(\eta_{\Omega})$ of the form $x=(k_1\eta_{\Omega},\ldots,k_n\eta_{\Omega})$ where the $k_i\in\N$ (see \cref{fig:main}).
	Then,
	\[
		c_L(X_{\Omega}) = \eta_{\Omega}.
	\]
\end{theorem} 
\begin{figure}[ht]
    \centering
    \begin{tikzpicture}
        \draw[->,color=black] (-0.5,0) -- (7,0);
        \draw[->,color=black] (0,-0.5) -- (0,5);
        \draw[-,color=blue, thick] (2,2) -- (2,4.5);
        \draw[-,color=blue, thick] (2,2) -- (6.5,2);
        \draw[dotted,color=black] (2,2) -- (0,2);
        \draw[dotted,color=black] (2,2) -- (2,0);
        \draw[color=black] (2,0) node[below] {\footnotesize{$\eta_\Omega$}};
        \draw[color=black] (0,2) node[left] {\footnotesize{$\eta_\Omega$}};
        \draw [black, thick] plot [smooth, tension=1] coordinates { (0,3) (1,4) (2,1.5) (4.1,1.96) (3.5,0.5) (4.5,0.4) (5.5,0.4) (6,0)};
        \draw[color=black] (1,1) node {\footnotesize{$\Omega$}};
        \draw[color=blue] (2,3) node[right] {\footnotesize{$\Omega_{N_n(\eta_\Omega)}$}};
        \draw[dotted,color=black] (4,2) -- (4,0);
        \draw[color=black] (4,0) node[below] {\footnotesize{$2\eta_\Omega$}};
        \node [red] at (4,2) {\footnotesize{\textbullet}};
        \draw[color=red] (4,2) node[above] {\footnotesize{$\Omega_T$}};
    \end{tikzpicture}
    \caption{Example of $X_{\Omega}$ satisfying the assumption in \cref{thm:generalformulacL}}
    \label{fig:main}
\end{figure}
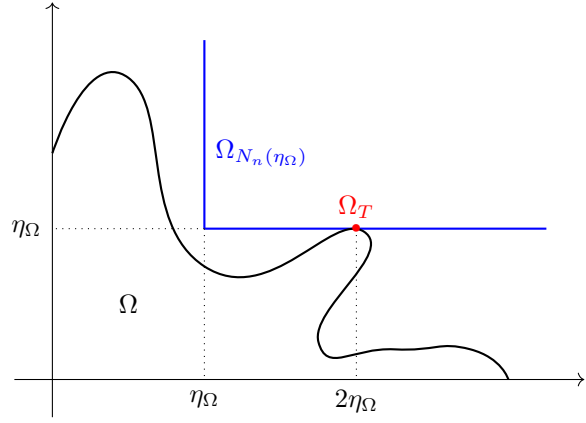

\section{An interesting nonexample}
We now study a family of examples coming from \cite{GHR} of non-monotone toric domains, and we determine when they satisfy the conclusion of \cref{thm:cube normalized}.

For $0<a<1/2$, let $\Omega_a$ be the convex polygon with corners $(0,0)$, $(1-2a,0)$, $(1-a,a)$, $(a,1-a)$ and $(0,1-2a)$, and write $X_a=X_{\Omega_a}$; see \cref{fig:example}. Then $X_a$ is a weakly convex (but not monotone) toric domain.
\begin{figure}[ht]
    \centering\label{fig:example}
	\begin{tikzpicture}[scale=2]
		\fill[orange!50](0,0)--(0,1.8)--(0.2,2)--(2,0.2)--(1.8,0)--(0,0);
		\draw[orange!50, dashed] (0,2.2)--(0.2,2);
		\draw[orange!50, dashed] (2,0.2)--(2.2,0);
		\draw [thick](0,0)--(0,2.3) ;
		\draw [thick] (0,0)--(2.3,0) ;
		\coordinate (A) at (1.6,0);
		\draw (A) node[below] {{$1-2a$}};
		\coordinate (B) at (2.2,0);
		\draw (B) node[below] {$1$};
		\filldraw (1.8,0) circle (1pt);
		\filldraw (2.2,0) circle (1pt);
		\filldraw (0,2.2) circle (1pt);
		\coordinate (C) at (0,1.8);
		\draw (C) node[left] {{$1-2a$}};
		\coordinate (D) at (0,2.2);
		\draw (D) node[left] {$1$};
		\filldraw (0,1.8) circle (1pt);
	\end{tikzpicture}
\caption{The domain $\Omega_a$}
\end{figure}
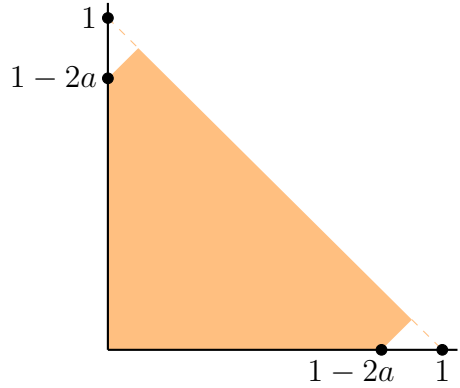

\begin{proposition}\label{prop:cpnXa}
The cubic, Lagrangian and NDUC capacities of $X_a$ are given as follows.
	\[\begin{aligned}c_P(X_a)&=\min\left(1-2a,\frac{1}{2}\right),\\
		c_L(X_a)=c^N(X_a)&=\frac{1}{2}.\end{aligned}
	\]
\end{proposition}
\begin{remark}
It follows from Proposition \ref{prop:cpnXa} that $c_P(X_a)\neq c^N(X_a)$ for $a> 1/4$. But in \cite{GHR} it was shown that $c_B(X_a)=c^Z(X_a)$ for all $a\le 1/3$. So for $1/4<a\le 1/3$, the Gromov and cylindrical capacities of $X_a$ coincide, but not the cubic and NDUC capacities.
\end{remark}

\begin{proof}
We note that $\eta_{\Omega_a}=1/2$ for all $a\le 1/2$ and that $(1/2,1/2)\in\Omega_a$. So it follows from Theorem \ref{thm:lag cap any toric} that $c_L(X_a)=1/2$. Since $X_a\subset N_2(a)$, it follows that
\[\frac{1}{2}=c_L(X_a)\le c^N(X_a)\le \frac{1}{2}.\]
So $c^N(X_a)=1/2$.

To compute the cubic capacities, we first observe that \[\begin{aligned}P_n\left(\frac{1}{2}\right)\subset X_a,&\text{ for }&0<a\le 1/4,\\ P_n(1-2a)\subset X_a,&\text{ for }&1/4 \le a <1/2.\end{aligned}\]
So $c_P(X_a)\ge \min(1-2a,1/2)$. Since $c_P(X_a)\le c^N(X_a)=1/2$, it follows that $c_P(X_a)=1/2=\min(1-2a,1/2)$ for $0<a\le 1/4$.

The fact that $c_P(X_a)\le 1-2a$ for $1/4<a<1/2$ follows from Theorem~\ref{thm:cubewc} below.
\end{proof}

\section{The cubic capacity of some weakly convex toric domains}
In this section we obtain an upper bound for the cubic capacity of some non-monotone toric domains, which will not in general coincide with their NDUC capacity.

A four-dimensional toric domain $X_\Omega$ is said to be weakly convex\footnote{Cristofaro-Gardiner defined this to be a convex toric domain in \cite{concaveconvex}, but usually a convex toric domain is defined to be a particular case of this, see \cite{GHR}, for example.} if $\Omega\subset\R^2_{\ge 0}$ is convex and $\partial_+\Omega$ is a piecewise smooth curve connecting the two coordinate axes, see Figure \ref{fig:wtc}. With an extra assumption, we can compute an upper bound for the cubic capacity of $X_\Omega$.

\begin{figure}\label{fig:wtc}
\centering
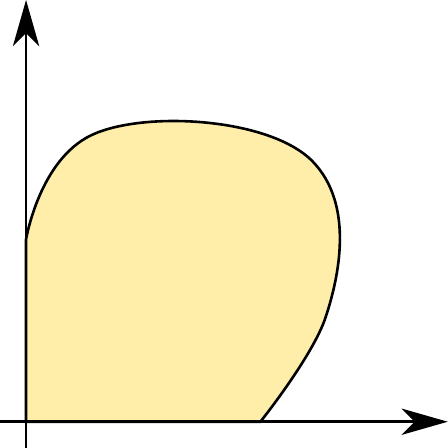
\caption{A weakly convex toric domain $X_\Omega$}
\end{figure}

\begin{theorem}\label{thm:cubewc}
Let $X_\Omega$ be a weakly convex toric domain, where $\partial_+\Omega$ is parametrized by the curve $(x,y):[0,1]\to\R^2_{\ge 0}$ such that $y(0)=0$ and $x(1)=0$. Suppose that \[\max\left(\frac{x'(0)}{y'(0)},\frac{y'(1)}{x'(1)}\right)\le 1.\]
Then \[c_P(X_\Omega)\le \frac{x(0)+y(1)}{2}.\]
\end{theorem} 

The proof of Theorem \ref{thm:cubewc} uses embedded contact homology. Namely, we need a version of \cite[Theorem 1.20]{beyond} for weakly convex toric domains. We now explain the context and the modifications that need to be made in the proof of \cite[Theorem 1.20]{beyond} for our purposes here. 

We need some definitions to state a more general version of \cite[Theorem 1.20]{beyond}. Let $X_\Omega$ be a weakly convex toric domain. We define a combinatorial Reeb orbit to be a pair $(v,s)$, where $v=(x_v,y_v)$ is a primitive vector in $\Z^2$ and $s=\{0,1\}$ such that $x_v\ge 0$ or $y_v\ge 0$. A combinatorial orbit set is a finite formal product
\[\alpha=\prod_{i=1}^k (v_i,s_i)^{m_i},\]
where $(v_i,s_i)$ are distinct combinatorial Reeb orbits and $m_i\in\Z_{\ge 1}$ such that $m_i=1$ whenever $s_i=0$. We define the following numbers.
\begin{align}
x(\alpha)&=\sum_{i=1}^k m_ix_{v_i},\label{eq:x}\\
y(\alpha)&=\sum_{i=1}^k m_iy_{v_i},\label{eq:y}\\
I(\alpha)&=x(\alpha)+y(\alpha)+\sum_{i,j=1}^k m_i m_j \max(x_{v_i} y_{v_j},x_{v_j} y_{v_i})+\sum_{i=1}^k s_i m_i,\label{eq:i}\\
m(\alpha)&=\sum_{i=1}^k m_i,\label{eq:m}\\
h(\alpha)&=\sum_{i=1}^k (1-s_i).\label{eq:h}
\end{align}
We note that none of those numbers depend on $\Omega$. The number $I(\alpha)$ is called the combinatorial ECH index of $\alpha$. We define the combinatorial action of $\alpha$ to be
\begin{equation*}
A_\Omega (\alpha)=\sum_{i=1}^k m_i \max\{v_i\cdot p\mid p\in\partial_+\Omega\}.
\end{equation*}
We now state a version of \cite[Definition 1.18]{beyond} for weakly convex toric domains.
\begin{definition}\label{def:beyond}
Let $X_{\Omega}$ and $X_{\Omega'}$ be weakly convex toric domains and let $\alpha$ and $\alpha'$ be combinatorial orbit sets. We write $\alpha\le_{\Omega,\Omega'} \alpha'$ if the following conditions hold:
\begin{itemize}
\item[(i)] $I(\alpha)=I(\alpha')$,
\item[(ii)] $A_\Omega(\alpha)\le A_{\Omega'}(\alpha')$,
\item[(iii)] $x(\alpha)+y(\alpha)-h(\alpha)/2\ge x(\alpha')+y(\alpha')+m(\alpha')-1$.
\end{itemize}
\end{definition}
The version of \cite[Theorem 1.20]{beyond} that we need is the following result.
\begin{theorem}\label{thm:beyond}
Let $X_\Omega$ and $X_{\Omega'}$ be weakly convex toric domains such that $X_\Omega\hookrightarrow X_{\Omega'}$. Let $\alpha'$ be an orbit set such that $I(\alpha')>0$ and $h(\alpha')=0$. Then there is an orbit set $\alpha$ with $I(\alpha)=I(\alpha')$ and product decompositions \[\alpha=\prod_{j=1}^l \alpha_j,\quad\alpha'=\prod_{j=1}^l \alpha_j',\] such that:
\begin{itemize}
\item[(a)] $\alpha_j\le_{\Omega,\Omega'} \alpha_j'$,
\item[(b)] Given $i,j$, if $\alpha_i=\alpha_j$ or $\alpha_i'=\alpha_j'$, then $\alpha_i$ and $\alpha_j$ have no combinatorial Reeb orbits in common with $s=1$.
\item[(c)] For any $\emptyset\neq S\subset\{1,\dots,l\}$, \[I\left(\prod_{j\in S} \alpha_j\right)=I\left(\prod_{j\in S} \alpha_j'\right)>0.\]
\end{itemize}
\end{theorem}

\begin{proof}
The proof is essentially the same as the one of \cite[Theorem 1.20]{beyond}. As in the proof of \cite[Theorem 5.6]{GHR}, we first approximate $\Omega$ by a domain $\widetilde{\Omega}\subset \Omega$ such that $\partial_+\widetilde{\Omega}$ is a smooth curve and the slopes of the tangent lines at the intersections with the $x$-axis and $y$-axis are $\varepsilon$ and $\varepsilon^{-1}$. We observe that for a given orbit set $\alpha$ and $\delta>0$, we can define $\widetilde{\Omega}$ so that $|A_{\Omega}(\alpha)-A_{\widetilde{\Omega}}(\alpha)|<\delta$. We define $\widetilde{\Omega}'\supset \Omega'$ satisfying the same properties as above, c.f. \cite[Lemma 5.4]{beyond}. In particular $X_{\widetilde{\Omega}}\hookrightarrow X_{\widetilde{\Omega}'}$.

We now briefly recall the embedded contact homology (ECH) chain complex. Let $(x,y):[0,1]\to\R^2$ be a parametrization of $\partial_+\widetilde{\Omega}$ such that $y(0)=x(1)=0$. So $y'(0)/x'(0)=x'(1)/y'(1)=\varepsilon$. We assume that $\varepsilon$ is a small irrational number and that $(x''(t),y''(t))\neq 0$ for $t\in[0,1]$. Then the standard Liouville form on $\R^4$ restricts to a contact form $\lambda_0$ on $\partial X_\Omega$ whose Reeb flow foliates $\mu^{-1}((x(t),y(t))$ for each $t\in[0,1]$. Then for each $t\in]0,1[$ such that $x'(t)/y'(t)\in\Q\cup\{\infty\}$, there is a unique $(p,q)\in\Z^2$ such that $p$ and $q$ are relatively prime and
\[(x'(t),y'(t))=c\cdot(p,q), \quad\text{for }c>0.\] So
the torus $T_{p,q}:=\mu^{-1}((x(t),y(t))$ is foliated by closed Reeb orbits. Note that $T_{(p,q)}$ is uniquely determined by $(p,q)$ since $X_\Omega$ is weakly convex. For a Reeb orbit $\gamma\in T_{p,q}$, its symplectic action is defined by
\[A_{\widetilde{\Omega}}(\gamma):=\int_\gamma \lambda_0.\] 
It is straight-forward to check that this action doesn't depend on $\gamma$. Indeed for every $\gamma\in T_{p,q}$, it follows from a simple calculation that
\[A_{\widetilde{\Omega}}(\gamma)=\max\{(p,q)\cdot x\mid x\in\partial_+ \widetilde{\Omega}\}=A_{\widetilde{\Omega}}((p,q),1).\] The only other Reeb orbits of $\lambda_0$ are the two circles $\mu^{-1}((x(0),y(0))$ and $\mu^{-1}((x(1),y(1))$. One can check that $\lambda_0$ is Morse--Bott. Given $L>0$, we can perturb the contact form in neighborhoods of the tori $T_{p,q}$ for which $A_{\widetilde{\Omega}}(\gamma)<L$ for $\gamma\in T_{p,q}$,
thus obtaining an elliptic and a hyperbolic Reeb orbit, denoted by $e_{(p,q)}$ and $h_{(p,q)}$, respectively. This is explained in more detail in \cite{qech} and \cite{concave}, for example. Let $\widetilde{\lambda}$ denote the pertubed contact form. The only other closed Reeb orbits of $\lambda$ with action less than $L$ are the two circles fibering above $(x(0),0)$ and $(0,y(1))$, which are elliptic. We denote them by $e_0$ and $e_1$.

An orbit set is a finite formal product
$\alpha=\prod_i \alpha_i^{m_i}$, where $\alpha_i$ is a simple Reeb orbit and $m_i$ is positive integer. We always assume that $\alpha_i\neq \alpha_j$ if $i\neq j$ and $m_i=1$ if $\alpha_i$ is hyperbolic. The action of an orbit set is defined by
\[A_{\widetilde{\Omega}}(\alpha)=\sum_i m_i A_{\widetilde{\Omega}}(\alpha_i).\] The filtered ECH chain complex $ECC^L(\partial X_{\widetilde{\Omega}},\widetilde{\lambda})$ is the $\Z/2$ vector space generated by all orbit sets $\alpha$ such that \[A_{\widetilde{\Omega}}(\alpha)<L.\]
Under the indentification $e_{p,q}=((p,q),1)$ and $h_{p,q}=((p,q),1)$, we can see orbit sets as combinatorial orbit sets and their symplectic actions coincide\footnote{To be precise, the symplectic actions with respect to the perturbed contact form is bounded from the combinatorial action by a small constant which can be as small as desired for a given $L$}.
The differential of $ECC^L(\partial X_{\widetilde{\Omega}},\widetilde{\lambda})$ is obtained by counting pseudo-holomorphic curves in $\R\times \partial X_{\widetilde{\Omega}}$ whose ECH index is 1. We will not define the ECH index here. Instead it suffices to recall that in this setting the ECH index gives rise to an absolute index such that for each orbit set $\alpha$, $I(\alpha)$ is simply the combinatorial ECH index defined in \eqref{eq:i}. The fact that the original definition and the combinatorial definition coincide follows from very similar calculations to the one in the proof of \cite[Lemma 5.4]{beyond}, which uses previous calculations from the proof of \cite[Lemma 3.3]{concave}. Here we have a max instead of a min, because of the opposite concavity, as in \cite[Lemma 5.4]{beyond}. It is worth noting that the calculation of the first Chern class \cite[(3.14)]{concave} is almost identical and in our case it gives \begin{equation}\label{eq:ctau}c_\tau(\alpha)=x(\alpha)+y(\alpha)\end{equation} as defined in \eqref{eq:x} and \eqref{eq:y}.

The rest of the argument uses the cobordism map in ECH and the $J_0$-invariant. It is identical to the proof of \cite[Theorem 1.20]{beyond} using \eqref{eq:ctau}, where we note that the original and the combinatorial definitions of $h$ and $m$ coincide.
\end{proof}
We can now prove Theorem \ref{thm:cubewc}.
\begin{proof}[Proof of Theorem {\ref{thm:cubewc}}]
Suppose that $P_2(a)\hookrightarrow X_\Omega$. We can find a weakly convex toric domain $X_{\Omega'}\supset X_\Omega$ such that the tangent lines to the curve $\partial_+ \Omega'$ at the $x$ and $y$ axes have slopes $1-\delta$ and $1+\delta$ for some small $\delta>0$, respectively. For each $L>0$ sufficiently large and $\varepsilon>0$, we can choose $X_{\Omega'}$ so that \begin{equation}\label{eq:ax0y1}|A_{\Omega'}(e_{1,-1})-x(0)|<\varepsilon\quad\text{ and }\quad |A_{\Omega'}(e_{-1,1})-y(1)|<\varepsilon,\end{equation} and that\[|A_{\Omega}(e_{p,q})-A_{\Omega'}(e_{p,q})|<\varepsilon,\]
for all $(p,q)$ such that $A_{\Omega}(e_{p,q})<L$.

Now let $\alpha'=e_{1,-1}^d e_{-1,1}^d e_{1,1}^2.$ It follows from Theorem \ref{thm:beyond} that there exists an orbit set $\alpha$ and factorizations
\[\alpha=\prod_{j=1}^l \alpha_j,\quad\alpha'=\prod_{j=1}^l \alpha_j',\]
satisfying (a), (b) and (c). Using (b) and (c), we conclude that $l\le 3$ and that $\alpha_i=e_{1,-1}^{d_i} e_{-1,1}^{d_i} e_{1,1}^k$ for some $k\in\{0,1,2\}$ such that $d_i\ge d/3$. Using (a), it follows from properties (ii) and (iii) from Definition \ref{def:beyond} that
\begin{equation*}
\begin{aligned}
3k+2d_i-1&=x(\alpha_i')+y(\alpha_i')+m(\alpha_i')-1\le x(\alpha_i)+y(\alpha_i)\\&=\frac{A_{P_2(a)}(\alpha_i)}{a}\le \frac{A_{\Omega'}(\alpha_i)}{a}<\frac{(d_i(x(0)+y(1))+k)(1+\varepsilon)}{a}. 
\end{aligned}
\end{equation*}
Hence \[a<\frac{(d_i(x(0)+y(1))+k)(1+\varepsilon)}{2d_i+3k-1}.\]
Taking the limit as $d\to\infty$ and then as $\varepsilon\to 0$, it follows that
\[a\le \frac{x(0)+y(1)}{2}.\]
Therefore
\[c_P(X_\Omega)\le \frac{x(0)+y(1)}{2}.\]
\end{proof}

\bibliographystyle{alpha}
\bibliography{bibliography}

\end{document}

%% file: wconvextoric.pdf_tex
\begingroup%
  \makeatletter%
  \providecommand\color[2][]{%
    \errmessage{(Inkscape) Color is used for the text in Inkscape, but the package 'color.sty' is not loaded}%
    \renewcommand\color[2][]{}%
  }%
  \providecommand\transparent[1]{%
    \errmessage{(Inkscape) Transparency is used (non-zero) for the text in Inkscape, but the package 'transparent.sty' is not loaded}%
    \renewcommand\transparent[1]{}%
  }%
  \providecommand\rotatebox[2]{#2}%
  \newcommand*\fsize{\dimexpr\f@size pt\relax}%
  \newcommand*\lineheight[1]{\fontsize{\fsize}{#1\fsize}\selectfont}%
  \ifx\svgwidth\undefined%
    \setlength{\unitlength}{128.95602856bp}%
    \ifx\svgscale\undefined%
      \relax%
    \else%
      \setlength{\unitlength}{\unitlength * \real{\svgscale}}%
    \fi%
  \else%
    \setlength{\unitlength}{\svgwidth}%
  \fi%
  \global\let\svgwidth\undefined%
  \global\let\svgscale\undefined%
  \makeatother%
  \begin{picture}(1,1.00001776)%
    \lineheight{1}%
    \setlength\tabcolsep{0pt}%
    \put(0,0){\includegraphics[width=\unitlength,page=1]{wconvextoric.pdf}}%
    \put(0.3572646,0.34480203){\color[rgb]{0,0,0}\makebox(0,0)[lt]{\lineheight{1.25}\smash{\begin{tabular}[t]{l}$\Omega$\end{tabular}}}}%
  \end{picture}%
\endgroup%